\documentclass[11pt]{article}
\usepackage{amsmath,amssymb}

\newtheorem{propo}{{\bf Proposition}}[section]
\newtheorem{coro}[propo]{{\bf Corollary}}
\newtheorem{lemma}[propo]{{\bf Lemma}} \newtheorem{theor}[propo]{{\bf
Theorem}} \newtheorem{ex}{{\sc Example}}[section]

\newenvironment{proof}{{\bf Proof.}}{$\Box$}

\def\Q{{\mathbb Q}}
\def\Z{{\mathbb Z}}
\def\N{{\mathbb N}}

\begin{document}

\vspace*{1.0in}

\begin{center} C-SECTIONS OF LIE ALGEBRAS 
\end{center}
\bigskip

\begin{center} DAVID A. TOWERS 
\end{center}
\bigskip

\begin{center} Department of Mathematics and Statistics
 
Lancaster University

Lancaster LA1 4YF

England

d.towers@lancaster.ac.uk 
\end{center}
\bigskip

\begin{abstract} Let $M$ be a maximal subalgebra of a Lie algebra $L$ and $A/B$ a chief factor of $L$ such that $B \subseteq M$ and $A \not \subseteq M$. We call the factor algebra $M \cap A/B$ a $c$-section of $M$. All such $c$-sections are isomorphic, and this concept is related those of $c$-ideals and ideal index previously introduced by the author. Properties of $c$-sections are studied and some new characterizations of solvable Lie algebras are obtained.
\end{abstract}

\section{Preliminary results}
Throughout $L$ will denote a finite-dimensional Lie algebra over a field $F$. We denote algebra direct sums by `$\oplus$', whereas vector space direct sums will be denoted by `$\dot{+}$'. If $B$ is a subalgebra of $L$ we define $B_L$, the {\em core} (with respect to $L$) of $B$ to be the largest ideal of $L$ contained in $B$. In \cite{cideal} we defined a subalgebra $B$ of $L$ to be a {\em c-ideal} of $L$ if there is an ideal $C$ of $L$ such that $L=B+C$ and $B \cap C \subseteq B_L$. 
\par

Let $M$ be a maximal subalgebra of $L$. We say that a chief factor $C/D$ of $L$ {\em supplements} $M$ in $L$ if $L=C+M$ and $D \subseteq C \cap M$; if $D=C \cap M$ we say that $C/D$ {\em complements} $M$ in $L$. In \cite{idealindex} we defined the {\em ideal index} of a maximal subalgebra $M$ of $L$, denoted by $\eta(L:M)$, to be the well-defined dimension of a chief factor $C/D$ where $C$ is an ideal minimal with respect to supplementing $M$ in $L$. Here we introduce a further concept which is related to the previous two.
\par

Let $M$ be a maximal subalgebra of $L$ and let $C/D$ be a chief factor of $L$ with $D \subseteq M$ and $L=M+C$. Then $(M \cap C)/D$ is called a {\em c-section} of $M$ in $L$. The analogous concept for groups was introduced by Wang and Shirong in \cite{w-s} and studied further by Li and Shi in \cite{l-s}.
\par

We say that $L$ is {\em primitive} if it has a maximal subalgebra $M$ with $M_L=0$. First we show that all c-sections of $M$ are isomorphic.

\begin{lemma}\label{l:unique} For every maximal subalgebra $M$ of $L$ there is a unique c-section up to isomorphism.
\end{lemma}
\begin{proof} Clearly c-sections exist. Let $(M \cap C)/D$ be a c-section of $M$ in $L$, where $C/D$ is a chief factor of $L$, $D \subseteq M$ and $L=M+C$. First we show that this c-section is isomorphic to one in which $D=M_L$. Clearly $D \subseteq M_L \cap C \subseteq C$, so either $M_L \cap C = C$ or $M_L \cap C=D$. If the former holds, then $C \subseteq M_L$, giving $L=M$, a contradiction. In the latter case put $E=C+M_L$. Then $E/M_L \cong C/D$ is a chief factor and $(M \cap E)/M_L$ is a c-section. Moreover, 
\[ \frac{M \cap E}{M_L} = \frac{M_L +M \cap C}{M_L} \cong \frac{M \cap C}{M_L \cap C} = \frac{M \cap C}{D}.
\]

So suppose that $(M \cap C_1)/M_L$ and $(M \cap C_2)/M_L$ are two c-sections, where $C_1/M_L$, $C_2/M_L$ are chief factors and $L=M+C_1=M+C_2$. Then $L/M_L$ is primitive and so either $C_1=C_2$ or else $C_1/M_L \cong C_2/M_L$ and $C_1 \cap M = M_L = C_2 \cap M$, by \cite[Theorem 1.1]{primitive}. In the latter case both c-sections are trivial. 
\end{proof}
\bigskip

Given a Lie algebra $L$ with a maximal subalgebra $M$ we define $Sec(M)$ to be the Lie algebra which is isomorphic to any c-section of $M$; we call the natural number $\eta^*(L:M) = \dim Sec(M)$ the {\em c-index} of $M$ in $L$.
\par

The relationship between c-ideals and c-sections, and between ideal index and c-index, for a maximal subalgebra $M$ of $L$ is given by the following lemma.

\begin{lemma} Let $M$ be a maximal subalgebra of a Lie algebra $L$. Then 
\begin{itemize}
\item[(i)] $M$ is a c-ideal of $L$ if and only if $Sec(M)=0$; and
\item[(ii)] $\eta^*(L:M) = \eta(L:M)- \dim(L/M)$.
\end{itemize}
\end{lemma}
\begin{proof}
\begin{itemize}
\item[(i)] Suppose first that $M$ is a c-ideal of $L$. Then there is an ideal $C$ of $L$ such that $L=M+C$ and $M \cap C \subseteq M_L$. Then $M \cap C = M_L \cap C$ is an ideal of $L$. Let $K$ be an ideal of $L$ with $M \cap C \subset K \subseteq C$. Then $K \not \subseteq M$, so $L=M+K$ and $M \cap C = M \cap K$. This yields that
$\dim L = \dim M + \dim K - \dim (M \cap K) = \dim M + \dim C - \dim (M \cap C),
$   so $K=C$ and $C/(M \cap C)$ is a chief factor of $L$. It follows that $Sec(M)=0$.
\par

The converse is clear.
\item[(ii)] Let $C/D$ be a chief factor such that $L=M+C$ and $C$ is minimal in the set of ideals supplementing $M$ in $L$. Then $\eta(L:M) = \dim(C/D)$, by the definition of ideal index. Thus,
\begin{align} \eta(L:M) & = \dim(C/D) = \dim C - \dim D \nonumber \\
                                      &  = \dim C - \dim C \cap M + \dim C \cap M - \dim D \nonumber \\
                                      &   = \dim L - \dim M + \dim(C \cap M/D)  \nonumber \\
                                      &  = \dim(L/M) + \eta^*(L:M).  \nonumber
\end{align}
\end{itemize}
\end{proof}

\begin{lemma}\label{l:supp} Let $A/B$ be an abelian chief factor of $L$. Then any maximal subalgebra of $L$ that supplements $A/B$ must complement $A/B$.
\end{lemma}
\begin{proof} Let $M$ supplement $A/B$, so $L=A+M$ and $B \subseteq M$. Then $[L,M \cap A]=[A+M,M \cap A] \subseteq B+M \cap A = M \cap A$. So $M \cap A$ is an ideal of $L$ and $M \cap A=B$.
\end{proof}
\bigskip

The following lemma will also be useful.

\begin{lemma}\label{l:factor} Let $B \subseteq M \subseteq L$, where $M$ is maximal in $L$ and $B$ is an ideal of $L$. Then $Sec(M) \cong Sec(M/B)$.
\end{lemma}
\begin{proof} Clearly $M/B$ is a maximal subalgebra of $L/B$. Let $(C/B)/(D/B)$ be a chief factor of $L/B$ such that $D/B \subseteq M/B$ and $C/B + M/B = L/B$. Then $C/D$ is a chief factor of $L$ such that $L=C+M$ and $D \subseteq M$. Hence $Sec(M) \cong C \cap M/D \cong Sec(M/B)$.
\end{proof}
\bigskip

 In \cite{primitive} it was shown that a primitive Lie algebra can be one of three types: it is said to be 
\begin{itemize}
\item[1.] {\em primitive of type $1$} if it has a unique minimal ideal that is abelian;
\item[2.] {\em primitive of type $2$} if it has a unique minimal ideal that is  non-abelian; and
\item[3.] {\em primitive of type $3$} if it has precisely two distinct minimal ideals each of which is  non-abelian.
\end{itemize}
If $M$ is a maximal subalgebra of $L$, then $L/M_L$ is clearly primitive; we say that $M$ is of type $i$ if $L/M_L$ is primitive of type $i$ for $i=1,2,3$. Then we have the following result.

\begin{lemma}\label{l:prim} Let $L$ be a Lie algebra over a field $F$ and let $M$ be a maximal subalgebra of $L$.
\begin{itemize}
\item[(i)] If $M$ is of type $1$ or $3$ then Sec$(M)=0$.
\item[(ii)] If $F$ has characteristic zero and $M$ is of type $2$ then Sec$(M) \cong M/M_L$.
\end{itemize}
\end{lemma}
\begin{proof} \begin{itemize}
\item[(i)] This follows from \cite[Theorem 1.1 3(a),(c)]{primitive}.
\item[(ii)] Let $A/B$ be a nonabelian chief factor that is supplemented by $M$, so $L=A+M$ and $B=A \cap M_L$. Then $L/M_L$ is simple, by \cite[Theorem 1.7 2]{primitive}, which implies that $L=A+M_L$. Hence
\[ \frac{M}{M_L} = \frac{M \cap (A+M_L)}{M_L} = \frac{M \cap A +M_L}{M_L} \cong \frac{M \cap A}{M_L \cap A} = \frac{M \cap A}{B} = Sec(M).
\]
\end{itemize}
\end{proof}

\section{Main results}
First we can state Theorems 3.1, 3.2 and 3.3 of \cite{cideal} in terms of c-sections as follows.

\begin{theor}\label{t:trivial}  Let $L$ be a Lie algebra over a field $F$. Then
\begin{itemize}
\item[(i)] every maximal subalgebra $M$ of $L$ has trivial c-section if and only if $L$ is solvable; and
\item[(ii)] if $F$ has characteristic zero, or is algebraically closed of characteristic greater than 5, then $L$ has a maximal subalgebra with trivial c-section if and only if $L$ is solvable.
\end{itemize}
\end{theor}

\begin{theor}\label{t:char0} Let $L$ be a Lie algebra over a field $F$ of characteristic zero. Then $Sec(M)$ is solvable for all maximal subalgebras $M$ of $L$ if and only if $L=R \dot{+} S$, where $R$ is the (solvable) radical of $L$ and $S$ is a direct sum of simple algebras which are minimal non-abelian or isomorphic to $sl_2(F)$.
\end{theor}
\begin{proof} Suppose first that $Sec(M)$ is solvable for all maximal subalgebras $M$ of $L$, and let $L=R \dot{+} S$ be the Levi decomposition of $L$. Then $Sec(M)$ is solvable for all maximal subalgebras $M$ of $S$, by Lemma \ref{l:factor}. Let $S = S_1 \oplus \ldots \oplus S_n$, where $S_i$ is simple for each $1 \leq i \leq n$. If $M$ contains all $S_i$ apart from $S_j$, then $Sec(M) \cong M \cap S_j$, so every subalgebra of $S_j$ is solvable. It follows from \cite[Theorem 2.2 and the remarks following it]{supsolv} that $S_j$ is minimal non-abelian or isomorphic to $sl_2(F)$ for each $1 \leq j \leq n$.
\par

Suppose conversely that $L$ has the claimed form and let $M$ be a maximal subalgebra of $L$. Every chief factor of $L$ is either abelian or simple, and so every c-section of $M$ is either abelian or isomorphic to a proper subalgebra of one of the simple components of $S$. In either case $Sec(M)$ is solvable. 
\end{proof}

\begin{coro}\label{c:char0}  Let $L$ be a Lie algebra over a field $F$ and suppose that every maximal subalgebra has $c$-index $k$. Then 
\begin{itemize}
\item[(i)] if $k >0$, $L$ must be semisimple.
\medskip

Suppose further that $F$ has characteristic zero. Then
\item[(ii)] every simple ideal of its Levi factor must have all of its maximal subalgebras of dimension $k$;
\item[(iii)] $k=0$ if and only if $L$ is solvable; 
\item[(iv)] $k=1$ if and only if $\sqrt{F} \not \subseteq F$ and $L$ is a direct sum of non-isomorphic three-dimensional non-split simple ideals; and
\item[(v)] $k=2$ if and only if $L$ is a direct sum of non-isomorphic ideals  and either (a) each of these ideals is a minimal non-abelian simple Lie algebra with all maximal subalgebras of dimension $2$, or (b) $\sqrt{F} \subseteq F$ and one of the ideals is isomorphic to $sl_2(F)$, whilst any others are minimal non-abelian simple Lie algebras with all maximal subalgebras of dimension $2$.
\end{itemize}
\end{coro}
\begin{proof} \begin{itemize}
\item[(i)] If $L$ has non-trivial radical, it has an abelian chief factor which is supplemented, and hence complemented, by Lemma \ref{l:supp}, so $k=0$.
\item[(ii)] This is clear.
\item[(iii)] This is Theorem \ref{t:trivial} (i).
\item[(iv)]  Suppose that $k=1$. Then $L$ is semisimple and each simple component has all of its maximal subalgebras one dimensional, by (i) and (ii). It follows that they are three-dimensional simple and $\sqrt{F} \not \subseteq F$, by \cite[Theorem 3.4]{chief}. Moreover, they must be non-split. If there are two that are isomorphic, say $S$ and $\theta(S)$, where $\theta$ is an isomorphism, then the diagonal subalgebra $\{s+ \theta(s): s \in S\}$ is maximal in $S \oplus \theta(S)$. But this together with the simple components other than $S$ and $\theta(S)$ gives a maximal subalgebra $M$ of $L$ with $c$-index $0$ in $L$. 
\par

Conversely, suppose that $L$ is a direct sum of non-isomorphic three-dimensional simple ideals, $S_1 \oplus \ldots \oplus S_n$, and $\sqrt{F} \not \subseteq F$. Let $M$ be a maximal subalgebra of $L$ with $S_i \not \subseteq M$ and $S_j \not \subseteq M$ for some $1 \leq i, j \leq n$ with $i \neq j$. Then $L=M+S_i=M+S_j$ which yields that $M \cap S_i$ and $M \cap S_j$ are ideals of $L$ and hence are trivial. But then $S_i \cong L/M \cong S_j$, a contradiction. It follows that every maximal subalgebra contains all but one of the simple components and hence that $k=1$.
\item[(v)] This is similar to (iv), using Theorem \ref{t:char0} and noting that $sl_2(F)$ has a one-dimensional maximal subalgebra if and only if $\sqrt{F} \not \subseteq F$, by \cite[Theorem 3.4]{chief}.
\end{itemize}
\end{proof}
\bigskip

Note that algebras as described in Corollary \ref{c:char0} do exist as the following example shows. This example was constructed by Gejn (see \cite[Example 3.5]{gejn}).
\begin{ex} Let $L$ be the Lie algebra generated by the matrices
\[ f_1 =  \left( \begin{array}{ccc}
0 & 0 & 0 \\
0 & 0 & -E \\
0 & E & 0 \end{array} \right),
f_2 = \left( \begin{array}{ccc}
0 & 0 & A \\
0 & 0 & 0 \\
-E & 0 & 0 \end{array} \right),
f_3 = \left( \begin{array}{ccc}
0 & -A & 0 \\
E & 0 & 0 \\
0 & 0 & 0 \end{array} \right)
\]
\[
g_1 =  \left( \begin{array}{ccc}
0 & 0 & 0 \\
0 & 0 & -A \\
0 & A & 0 \end{array} \right),
g_2 = \left( \begin{array}{ccc}
0 & 0 & 2E \\
0 & 0 & 0 \\
-A & 0 & 0 \end{array} \right),
g_3 = \left( \begin{array}{ccc}
0 & -2E & 0 \\
A & 0 & 0 \\
0 & 0 & 0 \end{array} \right)
\]
where $A= \left( \begin{array}{cc} 0 & 2 \\ 1 & 0 \end{array} \right)$, $E= \left( \begin{array}{cc} 1 & 0 \\ 0 & 1 \end{array} \right)$ and $0= \left( \begin{array}{cc} 0 & 0 \\ 0 & 0 \end{array} \right)$, with repect to the operation $[,]$, over the rational numbers $\Q$. Then $L$ is simple nonabelian (see \cite[Example 3.5]{gejn}),  and the maximal subalgebras are $\Q f_i + \Q g_i$ for $i=1,2,3$.
\end{ex}

\begin{ex}
Gejn also goes on to construct simple minimal nonabelian Lie algebras over $\Q$ of dimension $3k$ for $k \geq 1$ by putting 
\[ A= \left( \begin{array}{cccccc} 
0 & 0 & 0 & \ldots & 0 & 2 \\
1 & 0 & 0 & \ldots & 0 & 0 \\
0 & 1 & 0 & \ldots & 0 & 0 \\
. & . & . & \ldots & . & . \\
0 & 0 & 0 & \dots & 1 & 0 \end{array} \right)
\]
$E$ as the $k \times k$ identity matrix and $0$ as the $k \times k$ zero matrix (see  \cite[Example 3.6]{gejn}). It is straightforward to check that in these every maximal subalgebra has $c$-index $k$.
\end{ex}

The following corollary is straightforward.

\begin{coro}\label{c:onemax} Let $L=R\dot{+}S$ be a Lie algebra over a field $F$ of characteristic zero, where $R$ is the radical and $S$ is a Levi factor, and suppose that $L$ has a maximal subalgebra with $c$-index $k$. Then
\begin{itemize}
\item[(i)] if $k>0$ then $S \neq0$;
\item[(ii)] $k=1$ if and only if  $\sqrt{F} \not \subseteq F$ and $S$ has a minimal ideal $A$ which is three-dimensional simple;
\item[(iii)] $k>1$ if and only if $S$ has a minimal ideal with a maximal subalgebra of dimension $k$.
\end{itemize}
\end{coro}
\bigskip

Recall that a triple $(G,[p],\iota)$ consisting of a restricted Lie algebra $(G,[p])$ and a homomorphism $\iota : L \rightarrow G$ is called a {\em $p$-envelope} of $L$ if (a) $\iota$ is injective and (b) the $p$-algebra generated by $\iota(L)$ equals $G$. If $L$ is finite-dimensional then it has a finite-dimensional $p$-envelope (see, for example, \cite[Section 2.5]{s-f}). Let $(L_p,[p],\iota)$ be a $p$-envelope of $L$. If $S$ is a subalgebra of $L$ we denote by $S_p$ the restricted subalgebra of $L_p$ generated by $\iota(S)$. Then the {\em (absolute) toral rank} of $S$ in $L$, $TR(S,L)$, is defined by
\[ 
TR(S,L) = \hbox{max} \{\hbox{dim}(T) : T \hbox{ is a torus of } (S_p + Z(L_p))/Z(L_p)\}. 
\]
This definition is independent of the $p$-envelope chosen (see \cite{strade1}). We write $TR(L,L) = TR(L)$. A Lie algebra $L$ is {\em monolithic} if it has a unique minimal ideal (the {\em monolith} of $L$). The {\em Frattini ideal}, $\phi(L)$, is the largest ideal contained in every maximal subalgebra of $L$. We put $L^{(0)}=L$, $L^{(n)}=[L^{(n-1)},L^{(n-1)}]$ for $n \in \N$ and $L^{(\infty)}= \cap_{n=0}^{\infty} L^{(n)}$.

\begin{theor}\label{t:nilp} Let $L$ be a Lie algebra over an algebraically closed field $F$ of characteristic $p>0$. Then $Sec(M)$ is nilpotent for every maximal subalgebra $M$ of $L$ if and only $L$ is solvable. 
\end{theor}
\begin{proof} Let $L$ be a minimal non-solvable Lie algebra such that $Sec(M)$ is nilpotent for every maximal subalgebra $M$ of $L$, and let $R$ be the (solvable) radical of $L$. If $L$ is simple then every maximal subalgebra of $L$ is nilpotent, and no such Lie algebra exists over an algebraically closed field. So $L$ has a minimal ideal $A$, and $L/A$ is solvable. If there are two distinct minimal ideals $A_1$ and $A_2$, then $L/A_1$ and $L/A_2$ are solvable, whence $L \cong L/(A_1 \cap A_2)$ is solvable, a contradiction. Hence $L$ is monolithic with monolith $A$. If $A \subseteq R$ then again $L$ would be solvable, so $L$ is semisimple and $\phi(L)=0$. Thus, there is a maximal subalgebra $M$ of $L$ such that $L=M+A$.
\par

Put $C=M \cap A$ which is an ideal of $M$. If  ad\,$a$ is nilpotent for all $a \in A$ then $L$ is solvable, a contradiction. Hence there exists $a \in A$ such that ad\,$a$ is not nilpotent. Let $L=L_0 \dot{+} L_1$ be the Fitting decomposition of $L$ relative to ad\,$a$. Then $L_0 \neq L$ and  $L_1 \subseteq A$, so that if $P$ is a maximal subalgebra containing $L_0$, we have $L=A+P$ and $a \in A \cap P$. We can, therefore, assume that $C \neq 0$.
\par

Then $C$ is nilpotent and $L/A \cong M/C$ is solvable, whence $M$ is solvable. Now $[M,N_A(C)] \subseteq N_A(C)$, so $M+N_A(C)$ is a subalgebra of $L$. But $L=M+N_A(C)$ implies that $C$ is an ideal of $L$, from which $C=A$ and $L$ is solvable, a contradiction. It follows that $M=M+N_A(C)$, and so $N_A(C)=M \cap A = C$, and $C$ is a Cartan subalgebra of $A$. Now $C_p$ is a Cartan subalgebra of $A_p$, by \cite[Lemma]{wilson}, and so there is a maximal torus $T \subseteq A_p$ such that $C_p = C_{L_p}(T)$ (see \cite{seligman}).
\par

 Let $A_0(T)+ \sum_{i \in \Z_p} A_{i \alpha}$ be a $1$-section with respect to $T$. Then every element of $C$ acts nilpotently on $L_0$, the Fitting null-component relative to $T$, and thus so does every element of $C_p$. It follows that $L= L_0+ \sum_{i \in \Z_p} A_{i \alpha}$ so $L^{(\infty)}=A$ is simple with  $TR(A)=1$. We therefore have that 
\[p \neq 2, \hspace{.3cm} A  \in \{sl_2(F), W(1:\underline{1}), H(2:\underline{1})^{(1)}\} \hbox{ if } p>3
\]
\[ \hbox{and } \hspace{.3cm} A \in \{sl_2(F), psl_3(F)\} \hbox{ if } p=3,
\] 
by \cite{premet} and \cite{sk}. But now, $\dim A_{\alpha} = 1$ (by \cite[Corollary 3.8]{bo} for all but $psl_3(F)$, and this is straightforward to check) and $M=L_0 \subset L_0+A_{\alpha} \subset L$, a contradiction. It follows that $L$ is solvable.
\par

The converse is clear.
\end{proof}
\bigskip

A subalgebra $U$ of $L$ is {\em nil} if ad\,$u$ acts nilpotently on $L$ for all $u \in U$. Notice that we cannot replace `nilpotent' in Theorem \ref{t:nilp} by `solvable' or `supersolvable' and draw the same conclusion, as $sl_2(F)$ is a counter-example. However, we can prove the same result with `nilpotent' replaced by the stronger condition `nil' without any restrictions on the field $F$.

\begin{theor}\label{t:res} Let $L$ be a Lie algebra over any field $F$. Then $Sec(M)$ is nil for every maximal subalgebra $M$ of $L$ if and only if $L$ is solvable. 
\end{theor}
\begin{proof}  Let $L$ be a minimal non-solvable Lie algebra such that $Sec(M)$ is nil for every maximal subalgebra $M$ of $L$. If $L$ is simple then every maximal subalgebra of $L$ is nil. It follows that every element of $L$ is nil and $L$ is nilpotent, by Engel's Theorem. Hence no such Lie algebra exists. So, arguing as in paragraphs $1$ and $2$ of Theorem \ref{t:nilp} above, $L$ is monolithic with monolith $A$, $L/A$ is solvable, and there is a maximal subalgebra $M$ of $L$ such that $L=M+A$ with an element $a \in M \cap A$ such that ad\,$a$ is not nilpotent. But this is a contradiction, since $A \cap M=Sec(M)$ is nil.
\par

Once again, the converse is clear.
\end{proof}
\bigskip

Let $(L,[p])$ be a restricted Lie algebra. Recall that an element $x \in L$ is called {\em $p$-nilpotent} if there exists an $n \in \N$ such that $x^{[p]^n}=0$.  Then we have the following immediate corollary.

\begin{coro}\label{c:res} Let $L$ be a restricted Lie algebra over a field $F$ of characteristic $p>0$. Then $Sec(M)$ is $p$-nilpotent for every maximal subalgebra $M$ of $L$ if and only if $L$ is solvable. 
\end{coro}
\begin{proof} Simply note that that a $p$-nilpotent subalgebra is nil.
\end{proof}

\end{document}